\documentclass[12pt]{amsart}
\usepackage{mathtools} \mathtoolsset{showonlyrefs,showmanualtags}
\usepackage[
    colorlinks=true,       
    linkcolor=magenta,          
    citecolor=magenta,        
    filecolor=magenta,      
    urlcolor=magenta,           
    ]{hyperref}

\usepackage{stmaryrd,mathrsfs}
\usepackage{xcolor}
\numberwithin{equation}{section}
\allowdisplaybreaks[4]
\usepackage{amsmath,amsthm,amscd,amssymb,amsfonts, amsbsy}
\usepackage{latexsym}
\usepackage{exscale}

\newtheorem{Theorem}[equation]{Theorem}

\newtheorem{Corollary}[equation]{Corollary}
\newtheorem{Lemma}[equation]{Lemma}

\newtheorem{Remark}[equation]{Remark}

\newtheorem{Conjecture}[equation]{Conjecture}

\def\XXint#1#2#3{{\setbox0=\hbox{$#1{#2#3}{\int}$}
\vcenter{\hbox{$#2#3$}}\kern-.5\wd0}}

\def\bbZ{\mathbb{Z}}
\def\bbR{\mathbb{R}}

\newcommand{\strt}[1]{\rule{0pt}{#1}}

\begin{document}

\title[Proof of an extension of Sawyer's conjecture]{Proof of an extension of E. Sawyer's conjecture about weighted mixed weak-type estimates}

\author{Kangwei Li}
\address[K. Li]{BCAM, Basque Center for Applied Mathematics, Mazarredo, 14. 48009 Bilbao Basque Country, Spain}
\email{kli@bcamath.org}
\author{Sheldy Ombrosi}
\address[S. Ombrosi]{Department of Mathematics, Universidad Nacional del Sur, Bah\'ia Blanca, Argentina}
\email{sombrosi@uns.edu.ar}
\author{Carlos P\'erez}
\address[C. P\'erez]{Department of Mathematics, University of the Basque Country, Ikerbasque and BCAM, Spain}
\email{carlos.perezmo@ehu.es}

\thanks{K.L. and C.P.  are supported by the Basque Government through the BERC 2014-2017 program and by Spanish Ministry of Economy and Competitiveness MINECO: BCAM Severo Ochoa excellence accreditation SEV-2013-0323. , S. O. is supported by CONICET PIP 11220130100329CO, Argentina. C.P. is  supported by Spanish Ministry of Economy and Competitiveness MINECO through the project MTM2014-53850-P}


\keywords{Sawyer's conjecture, weak type inequalities, $A_1$ and $A_{\infty}$ weights}
\subjclass[2010]{Primary: 42B25, Secondary: 42B20}

\begin{abstract}

We show that  if $v\in A_\infty$ and $u\in A_1$, then there is a constant $c$ depending on the $A_1$ constant of $u$ and the $A_{\infty}$ constant of $v$ such that
$$\Big\|\frac{ T(fv)} {v}\Big\|_{L^{1,\infty}(uv)}\le c\, \|f\|_{L^1(uv)},$$
where $T$ can be the Hardy-Littlewood maximal function or any Calder\'on-Zygmund operator.  This result was conjectured in [IMRN, (30)2005, 1849--1871]  and  constitutes the most singular case of some extensions of several problems proposed by E. Sawyer and Muckenhoupt and Wheeden.  We also improve and extends several quantitative estimates.
\end{abstract}

\maketitle

\section{Introduction and main results}

The purpose of this paper is to prove some extensions of several conjectures formulated by E. Sawyer in \cite{Sa} where it is proved the following weighted weak type inequality for the Hardy-Littlewood maximal function on the real line: if  $u, v \in  A_1$, then
\begin{equation} \label{Sawyer}
\Big\| \frac{ M(fv)} {v}\Big\|_{L^{1,\infty}(uv)}\le c_{u,v} \|f\|_{L^1(uv)}.
\end{equation}
This estimate is a highly non-trivial extension of the classical weak type $(1, 1)$
inequality for the maximal operator due to the presence of the weight function $v$ inside the distribution set. These type of estimates are also referred as mixed weak-type estimates. Observe that if $v=1$ this result is the well known estimate due to C. Fefferman-E. Stein theorem \cite{FS1} which holds if and only if  $u\in A_1$.  Also, if $u=1$, then the result also holds when $v\in A_1$ by a simple argument. However, in this case the $A_1$ condition is not necessary and many other examples can be found (see below).

In the general situation this estimate becomes more difficult. There are basically two main obstacles. The first problem is that the product $u v$  may be very singular.  For instance, let $u(x) = v(x) = |x|^{-1/2}$ on the real line. Then, $u$ and $v$ are $A_1$ weights but the product $uv$  is not locally integrable. The second drawback is that  the operator $f\to \frac{M(fv)}{v}$, which can be seen as a perturbation of $M$ by the weight $v$, changes dramatically the level sets of $M$.
For instance,  it is not clear how to apply directly any covering  lemma to $\{ \frac{M(fv)}{v}>t\}$, specially in the case $v\in A_{\infty}$.

The main motivation in \cite{Sa} to prove \eqref{Sawyer} is due to the fact that
it yields a new proof of the classical $A_p$ theorem of Muckenhoupt \cite{M} assuming the factorization theorem for $A_p$ weights (see \cite{J}). However, there are many situations where these mixed weak type inequalities appear naturally (see below). In the same paper, E. Sawyer conjectured that \eqref{Sawyer} should hold for the Hilbert
transform and for $M$ but in higher dimensions. A positive answer to both conjectures was found in \cite{CMP2}.  In fact these conjectures were further extended in several directions, first to the case of Calder\'on-Zygmund operators (even to the case of the maximal singular integral) and second to a larger class of weights solving at once some other interesting conjectures formulated earlier by Muckenhoupt and Wheeden in \cite{MW}.

The method of proofs in \cite{CMP2} is based on the combination of the following facts or results: \\
1)  An extrapolation type result for $A_{\infty}$ weights  in the spirit of the results obtained in \cite{CMP1} and \cite{CGMP}. \\
2) The use of the Rubio de Francia method \cite{CUMP} within the context of Lorentz spaces. \\
3) The use of the R. Coifman-C. Fefferman estimate (see \eqref{coifman-fefferman} below) relating singular integrals and the maximal function  in the $L^p(w)$ spaces  with $p\in (0,1)$, never used in the literature before, and for the whole class of weights $A_{\infty}$, not just $A_p$ (see \cite{CoF}).\\
4) Reduction the problem from singular integral operators to the dyadic maximal function.

Further extensions of the conjectures formulated by E. Sawyer (and also by Muckenhoupt-Wheeden \cite{MW})  were raised in \cite{CMP2}. The most difficult one  of these conjectures is to prove \eqref{Sawyer} assuming that $v\in A_{\infty}$ since it is  the  largest class.
Although some improvements were made later in  \cite{OP}, and some more precise quantitative estimates were obtained in \cite{OPR}, the full conjecture remained open. The main purpose of this paper is to prove this conjecture and to derive some consequences.

Our main result is the following.
\begin{Theorem}\label{thm:main}   Let $M$ be the Hardy-Littlewood maximal operator on $\mathbb R^n$
and let $u\in A_1$ and $v\in A_\infty$.  Then there is a finite constant $c$ depending on the $A_1$ constant of $u$ and the $A_{\infty}$ constant of $v$ such that
\begin{equation} \label{BigConjecture}
\Big\|\frac{ M(fv)} {v}\Big\|_{L^{1,\infty}(uv)}\leq c \|f\|_{L^1(uv)}.
\end{equation}

\end{Theorem}

We point out some observations.

$\bullet$ We will see that it is enough to consider the dyadic maximal function instead of $M$.

$\bullet$ The conditions $u\in A_1$ and $v\in A_\infty$ are weaker than $u\in A_1$ and $uv \in A_\infty$ (which is one of the assumptions in \cite{CMP2}). Indeed, when $u\in A_1$, $uv\in A_\infty$ is equivalent with $v\in A_\infty (u)$ (see \cite[Remark 2.2]{CMP2}) and the latter implies
\[
\frac{v(E)}{v(Q)}\le C \Big(\frac{u(E)}{u(Q)}\Big)^\varepsilon\le C \Big(\frac{|E|}{|Q|}\Big)^{\varepsilon\delta},\,\, E\subset Q.
\]

$\bullet$ If we let $v=\frac{1}{ |x|^{nr}}$, $r>1$, then the  inequality above is true even for $(u,Mu)$, namely
\begin{equation} \label{uMu}
\Big\|\frac{ M(fv)} {v}\Big\|_{L^{1,\infty}(uv)}\leq c \|f\|_{L^1(vMu)}     \qquad u\geq 0
\end{equation}
as can be found in \cite{OP} (see also \cite{MOS}). Observe that no assumption is made on $u$.

$\bullet$ Surprisingly, \eqref{uMu} is false in general if $v\in A_{\infty}$ as can be found in \cite[Example 3.1]{OP}. Further, it is false if $Mu$ is replaced by $M^2u$.  Again, no assumption is made on $u$.

$\bullet$ However, \eqref{uMu} is false when $r=1$ even in the case $u=1$, see \cite{AM}.

$\bullet$ Even further, if $u=1$  and $v=M\mu$,  where $\mu$ is a non-negative function or measure, \eqref{BigConjecture} is false as can be found in \cite[Section 5]{OPR}.  Recall that $v=M\delta \approx \frac{1}{ |x|^{n}}$, where $\delta$ is the  Dirac measure.   This shows that for the class of weights $v$ such that $v^{\epsilon} \in A_{\infty}$ for some small $\epsilon$,  \eqref{BigConjecture} is false in general.  However, it was proved in \cite{OP} that the key extrapolation theorem, similar to Lemma \ref{extrapolation} below,  holds for these class of weights.

$\bullet$ If $uv=1$ then the result is true whenever $v\in A_1$, but it is false in general when $v\in A_p \setminus A_1$, $p>1$, see \cite{PR}.

In view of these positive and negative  examples  we establish the following conjecture.

\begin{Conjecture} Let $v$ be a weight such that
\[
\Big\|\frac{ M(fv)} {v}\Big\|_{L^{1,\infty}(v)}\le c_{v} \|f\|_{L^1(v)}.
\]
Then, there is a finite constant $c$ depending on the $A_1$ constant of $u$ and $c_v$ above such that
\[
\Big\|\frac{ M(fv)} {v}\Big\|_{L^{1,\infty}(uv)}\le c\, \|f\|_{L^1(uv)} \qquad u\in A_1.
\]
\end{Conjecture}

We already mentioned earlier that the main motivation to study these type of estimates in \cite{Sa} is to produce a new proof of the boundedness of the maximal operator on $L^p(w)$,  $w \in A_p$, assuming the factorization theorem of $A_p$ weights. However, there are several other interesting applications:

$\cdot$ {\it Multilinear estimates.} It was shown in \cite{LOPTT} that the multilinear operator defined by  $\prod_{j=1}^m Mf_j$ is too big to be considered as the maximal operator controlling the multilinear Calder\'on-Zygmund operators. Instead, the pointwise smaller maximal  operator $\mathcal M$ introduced in that paper is the correct one (we remit the reader to the paper for the definition). Nevertheless, this operator is interesting on its own and it was shown there that it does satisfy sharp weighted weak-type estimates whose proof is based on the mixed
weak-type inequalities derived in \cite{CMP2}.  To be more precise, if $\{w_{i}\}^m_{i=1}$ is a family of weights such that $w_{i} \in A_{1}$, for all $i=1,2,...,m$, and if
$\nu=\Big(\prod_{j=1}^mw_j\Big)^{1/m},$ then
\begin{equation}\label{multiweak}
\|\prod_{j=1}^mMf_{j}\|_{L^{\frac{1}{m},\infty }(\nu)}\le
C\prod_{j=1}^m\|f_{i}\|_{L^{1}(w_i)}.
\end{equation}
which is an extension of the classical result mentioned above due to C. Fefferman-E. Stein,
\begin{equation*}
\|Mf \|_{L^{1,\infty}(u)}\leq c\, \|f\|_{L^{1}(u)}   \quad \mbox{if and only if} \quad u\in A_1.
\end{equation*}
The strong version of this result, namely
$$
\|\prod_{j=1}^mMf_j \, \|_{ \strt{1.5ex}  L^{p}(\nu_{\vec w})}\leq
c\,\prod_{j=1}^m\|f_j\|_{ \strt{1.5ex} L^{p_j}(Mw_j)},
$$
where $\nu_{\vec w}=\prod_{j=1}^mw_j^{p/p_j}$ with $\frac{1}{p}=\frac{1}{p_1}+\dots+\frac{1}{p_m}$, $p_i\in(1,\infty)$, $i=1,\cdots,m$,
can be obtained directly from the multilinear H\"older's inequality and the classical Fefferman-Stein's inequality.
However, we cannot  repeat the same argument for the weak norm result \eqref{multiweak} and therefore a proof is required.
Indeed, this result is more difficult since we have to control the distribution set: (in the case $m=2$)
$$\nu\bigg\{x \in \mathbb R^n:  Mf_1(x)Mf_2(x)> 1  \bigg\} = \nu\bigg\{x \in \mathbb R^n:  Mf_1(x)> \frac{1}{Mf_2(x)}  \bigg\}.$$
The key observation is that $(Mf)^{-1}\in A_{\infty}$ \, with constant independent of $f$ and hence we are dealing with an estimate that fits within our context.

$\cdot$ {\it Commutators with BMO functions.} $L^p$ estimates of the commutators of Coifman-Rochberg-Weiss $[b,T]$ can be derived in a very effective way by means of the conjugation method considered in \cite{CRW}. This method works when $T$ is a bounded linear operator on $L^2(w)$, $w \in A_2$
and $b\in BMO$. However, this method breaks down  when considering endpoint estimates. Indeed, the conjugation method is intimately related to the family operators $\{T_w\}_{w \in A_p}$, $p>1$, of the form
\begin{equation*}
f
\to   T_w(f) := w\,T\left(\frac{f}{w}\right)
\end{equation*}
These families of operators have the same structure as the ones we are considering in the present paper. We could consider the question of whether they are of weak type $(1,1)$ or not.  In fact, it is shown in \cite{PR} that these operators, in the case for instance of the Hilbert transform, are not of weak type $(1,1)$ in general with respect to the Lebesgue measure. This behavior is the same as in the case of commutators, as was shown in \cite{P}. On the other hand, these operators $T_w$, when $T$ is a Calder\'on-Zygmund operator, are of weak type $(1,1)$ when $w\in A_1$. However, the $A_1$ class of weights is excluded from the method of conjugation.

There is also another interesting connection between weighted mixed  weak type inequalities and Ergodic Theory. We remit the reader to Section 4 in \cite{Ma} for details.

As a corollary of the main theorem \ref{thm:main} we obtain the following result.

\begin{Corollary}\label{Cor1}
Let $M$ be the Hardy-Littlewood maximal operator and let $T$ be any operator such that for some $p_0 \in (0,\infty)$ and for any $w\in A_\infty$, there is a constant $c$ depending on the $A_{\infty}$ constant of $w$  such that,
\begin{equation}\label{coifman-fefferman}
\int_{\mathbb{R}^{n}} |Tf(x)|^{p_0}\, w(x)\,d x
\leq
c\,\int_{\mathbb{R}^{n}} Mf(x)^{p_0}\, w(x)\,d x,
\end{equation}
for any function $f$ such that the left hand side is finite.  Then the analogue of \eqref{BigConjecture} holds for $T$, namely if $v\in A_\infty$ and $u\in A_1$, there is a constant $c$ depending on the $A_1$ constant of $u$ and the $A_{\infty}$ constant $v$ such that,
\[
\Big\|\frac{ T(fv)} {v}\Big\|_{L^{1,\infty}(uv)}\le c\, \|f\|_{L^1(uv)}.
\]
\end{Corollary}

The proof of the corollary is a consequence of the following extrapolation type result for $A_{\infty}$ weights that can be found in
\cite[Theorem 1.7]{CMP2}.

\begin{Lemma}\label{extrapolation}
Let  $\mathcal F$ be a family of ordered pairs of non-negative, measurable functions $(f,g)$. Let $p_0\in (0,\infty)$ such that for every $w\in A_\infty$,
\begin{equation}\label{extraphyp}
\|f\|_{L^{p_0}(w)}\le c \|g\|_{L^{p_0}(w)}
\end{equation}
for all $(f,g)\in \mathcal F$ such that the left-hand side is finite, and where $c$ depends only on
the $A_\infty$ constant of $w$. Then for all weights $u\in A_1$ and $v\in A_\infty$,
\[
\Big\| \frac{f}{v}\Big\|_{L^{1,\infty}(uv)}\leq c \Big\| \frac{g}{v}\Big\|_{L^{1,\infty}(uv)}, \qquad(f,g)\in \mathcal F,
\]
with a constant $c$ depending on the $A_1$ constant of $u$ and the $A_{\infty}$ constant of $v$.
\end{Lemma}

This result was extended in \cite{OP} to a larger class of weights $v$, namely those weights such that for some $\delta>0$, $v^{\delta} \in A_{\infty}$.   However, as we already mentioned, it is not true that for every weight in this class, Theorem \ref{thm:main} holds.

Corollary \ref{Cor1} also applies to the following operators: \\
1) To any Calder\'on-Zygmund operators. This follows from the classical good-$\lambda$ inequality between $T$ and $M$ due to R. Coifman-C. Fefferman  \cite{CoF}. \\
2) To any rough singular integral $T_{\Omega}$ with $\Omega \in L^{\infty}(\mathbb{S}^{n-1}).$ \\
3) To the Bochner-Riesz multiplier at the critical index. \\
There is no available  good-$\lambda$ inequality between any of these operators and $M$. However, an estimate like \eqref{coifman-fefferman} holds for these operators as we have proved recently in \cite{LPRR} for any $p_0>0$. \\
4) Square functions of the form $g_\lambda^*(f)$  (see \cite{St} for the definition).    In fact, \eqref{coifman-fefferman} follows from
\begin{equation} \label{puntual-g-function}
M^{\#}_{\delta}(g_\lambda^*(f))(x)\le C_{\delta,\lambda}\,  Mf(x) \qquad  \lambda>2, 0<\delta<1,
\end{equation}
which can be found in \cite{CP}, together with the C. Fefferman-Stein estimate \cite{FS2} (see also \cite{Duo}):
\begin{equation} \label{FS-sharpMF}
\int_{\mathbb R^n} |f(x)|^p\, w(x)\,d x
\le
c\,\int_{\mathbb R^n} M^\# f(x)^p\, w(x)\,d x,
\end{equation}
for any $A_\infty$ weight $w$, any $p$, $0<p<\infty$ and for any
function $f$ such that left hand side is finite.

Another consequence of Theorem \ref{thm:main} is the following vector-valued extension.

\begin{Corollary}\label{Cor2} Let $T$ be any operator satisfying the hypotheses from Corollary \ref{Cor1} above and let $u\in A_1$ and $v \in A_{\infty}$.
Let $q\in (1,\infty) $ then the following vector--valued extension holds:
there is a constant $c$ depending on the $A_1$ constant of $u$ and the $A_{\infty}$ constant $v$,
\begin{equation} \label{VVBigConjecture}
\bigg\|    \frac{\Big(\sum_j
|T(vf_{j})|^q\Big)^{\frac1q}}{v}       \bigg\|_{L^{1,\infty}(uv)}
\le c\, \big\|   \Big(    \sum_j
|f_j|^q\Big)^\frac1q  \big\|_{L^1(uv)}.
\end{equation}

\end{Corollary}

\section{Proof of Theorem \ref{thm:main}}
This section is devoted to prove Theorem \ref{thm:main}. We still follow the general idea of Sawyer \cite{Sa}. However, Sawyer's proof depends heavily on $v\in A_1$.  To overcome this difficulty we combine the `pigeon-hole' technique, the Calder\'on-Zygmund decomposition, and the two key Lemmas \ref{lm:sparse} and \ref{lm:decay} below.

We start by using the well known fact that
\[
M(f)\le c_n \sum_{i=1}^{3^n} M_{\mathscr D^{(i)}}f,
\] where $\mathscr D^{(i)}$ is a dyadic system for all $1\le i \le 3^n$. So it is enough to obtain the following inequality for any $g$ bounded with compact support
\begin{equation}\label{eq:muv}
uv\{x\in \mathbb R^n:1< \frac{M_d(g)(x)}{v(x)}\le 2\}\le C_{u,v}\int_{\bbR^n}gu dx,
\end{equation}
where $M_d := M_{\mathscr D}$ and $\mathscr D$ is one of the dyadic system $\mathscr D^{(i)}$, $1\le i \le 3^n$. We decompose the left hand side of \eqref{eq:muv} as
\[
\sum_{k\in \mathbb Z}uv\{x\in \mathbb R^n:1< \frac{M_d(g)(x)}{v(x)}\le 2, a^k < v(x)\le a^{k+1}\}:=\sum_{k\in \mathbb Z}uv(E_k) ,
\]
where $a>2^n$. For each $k$ we define
\[
\Omega_k:= \{x\in \bbR^n: M_d(g)(x)>a^k\}.
\]
Observe that $E_k\subset \Omega_k$. Let $\mathcal Q_k := \{I_j^k\}$ denote the collection of maximal, disjoint dyadic cubes whose union is $\Omega_k$. By maximality,
$$a^k<\langle g\rangle_{I_{j}^k}\le 2^n a^k.
$$
We split now the family of cubes $\{I_j^k\}_j$ as
\[
\mathcal Q_{l,k} :=  \{I_j^k\in \mathcal Q_k: a^{k+l} \leq  \langle v\rangle_{I_j^k} < a^{k+l+1}\}, \quad l\ge 0
\]
and
\[
\mathcal Q_{-1,k} :=  \{I_j^k\in \mathcal Q_k: \langle v\rangle_{I_j^k} < a^{k}\}.
\]

Now, for a fixed $I_j^k\in \mathcal Q_{-1,k}$, since $a^k>\langle v\rangle_{I_j^k}$, we form the Calder\'on-Zygmund decomposition to $v\chi^{}_{I_j^k}$ at height $a^k$. Hence, we obtain a collection of subcubes  $\{I_{j,i}^k\}_i\in \mathcal D(I_j^k)$ which satisfy
\begin{equation} \label{microlocalCZ}
a^k<\langle v\rangle_{I_{j,i}^k}<2^n a^k, \quad \forall \, i.
\end{equation}
Moreover,
$$v(x)\le a^k,  \qquad x \in I_j^k \setminus \bigcup_i I_{j,i}^k.$$
Now denote $\Omega_{-1,k}=\bigcup_{I_j^k\in \tilde\Omega_{-1,k}} \bigcup_i I_{j,i}^k $.
We have
\begin{align*}
\sum_{k\in \mathbb Z}uv(E_k)&=\sum_{k\in \mathbb Z}uv(E_k\cap\Omega_k)= \sum_{k\in \mathbb Z}\sum_{j}uv(E_k\cap I_{j}^k)  =
\\
&\le \sum_{k\in \bbZ}\sum_{l\ge 0}\sum_{I_j^k\in \mathcal Q_{l,k}} a^{k+1}u(E_k\cap I_j^k)
+  \sum_{k\in \bbZ}\sum_{I_{j}^k\in \mathcal Q_{-1,k}} a^{k+1}u(E_k\cap I_j^k)\\
&\le \sum_{k\in \bbZ}\sum_{l\ge 0}\sum_{I_j^k\in \Gamma_{l,k}} a^{k+1}u(E_k\cap I_j^k)
+ \sum_{k\in \bbZ} \, \sum_{i: I_{j,i}^k\in \Gamma_{-1,k}}  a^{k+1}u(I_{j,i}^k),
\end{align*}
where
\[
\Gamma_{l,k}=\{I_{j}^k\in \mathcal Q_{l,k}: |I_j^k\cap \{x: a^k<v\le a^{k+1}\}|>0\}\quad \mbox{if $l \ge 0$},
\]
and
\[
\Gamma_{-1,k}=\{I_{j,i}^k\in \mathcal Q_{-1,k}: |I_{j,i}^k\cap \{x: a^k<v\le a^{k+1}\}|>0\}.
\]

In the following, we shall deal with the case $l=-1$ and $l \ge 0$ separately. By monotone convergence theorem, it suffices to give a uniform estimate for
\[
\sum_{k\ge N} \sum_{l\ge 0}\sum_{I_j^k\in \Gamma_{l,k}} a^{k+1}u(E_k\cap I_j^k)+ \sum_{k\ge N}\,
\sum_{i:I_{j,i}^k\in \Gamma_{-1,k}}  a^{k+1}u(I_{j,i}^k),
\]
where $N<0$. The following lemma is a key.

\begin{Lemma}\label{lm:sparse}
$\Gamma=\cup_{ l\ge -1}\cup_{k\ge N}\Gamma_{l,k}$ is sparse.
\end{Lemma}

\begin{proof}
First, we prove that if $I_j^k\subsetneq I_s^t$ or $I_{j,i}^k\subsetneq I_s^t$ or $I_j^k\subsetneq I_{s,\ell}^t$ or $I_{j,i}^k\subsetneq I_{s,\ell}^t$, then $k>t$. The case $I_j^k\subsetneq I_s^t$ is obvious since they are maximal dyadic cubes in $\mathcal Q_k$ and $\mathcal Q_t$, respectively. For the case $I_j^k\subsetneq I_{s,\ell}^t$, then again $I_j^k\subsetneq I_s^t$. For the case $I_{j,i}^k\subsetneq I_s^t$, obviously $I_j^k\neq I_s^t$. Notice that if $I_s^t \subsetneq I_j^k$, then $t>k$ and
\[
\langle v\rangle_{I_s^t} >a^t>a^k
\]
which means there is some cube $Q$ such that $I_{j,i}^k\subsetneq Q$ and $a^k<\langle v\rangle_{Q}<2^na^{k}$, a contradiction! Finally, if $I_{j,i}^k\subsetneq I_{s,\ell}^t$, assume $t>k$, then $I_s^t\subsetneq I_j^k$. Therefore
\[
\langle v\rangle_{I_{s,\ell}^t}>a^t>a^k,
\]and $I_{s,\ell}^t\subsetneq I_j^k$, again, a contradiction. So we have proved the first claim. With this claim at hand, it is easy to check that
\[
|\bigcup_{\substack{Q',Q\in \Gamma\\ Q'\subsetneq Q}}Q'|\le \frac{2^n}{a}|Q|.
\]
\end{proof}

\subsection{The case $l\ge 0$}
First we need the following lemma.

\begin{Lemma}\label{RHA1} Let $w\in A_\infty$, and let
$r_{w}=1+\frac{1}{\tau_n [w]_{A_\infty}}$. Then for any cube $Q$
$$ \left(\frac{1}{|Q|}\int_Q w^{r_w}\right)^{1/r_w} \le \frac{2}{|Q|}\int_Q
w.$$
As a consequence we have that for any cube $Q$ and for any measurable set $E\subset Q$
\begin{equation*}
 \frac{w(E)}{w(Q)} \leq 2 \left(\frac{|E|}{|I|}\right)^{\epsilon_w },
 \end{equation*}
where $\epsilon_w =\frac1{1+\tau_n[w]_{A_\infty} }. $
\end{Lemma}

The proof of this reverse H\"older inequality can be found in \cite{HyPe} and the consequence is an application of H\"older's inequality.

Another key point is the following lemma.

\begin{Lemma}\label{lm:decay}
For $l\ge 0$ and $I_j^k\in \Gamma_{l,k}$, there exist constants $c_1$ and $c_2$ depends on $u,v$ such that
\[
  u(E_k\cap I_j^k)\le c_1  e^{- c_2 l} u(I_j^k).
\]
\end{Lemma}
\begin{proof}
Since $v \in A_\infty$, by embedding, we know that there exists some $q$ such that $v\in A_q$. Then
\[
\Big(\frac{|E_k\cap I_j^k|}{|I_j^k|}\Big)^{q-1}a^{-k-1}\le \Big(\frac{1}{|I_j^k|}\int_{I_j^k}v^{-\frac 1{q-1}}\Big)^{q-1}\le \frac{[v]_{A_q}}{\langle v \rangle_{I_j^k}}\le a^{-k-l}[v]_{A_q}.
\]
It follows that
\[
\frac{|E_k\cap I_j^k|}{|I_j^k|}\le a_{}^{\frac{1-l}{q-1}}[v]_{A_q}^{\frac 1{q-1}}.
\]
Then since $u\in A_1$ we can use Lemma \ref{RHA1}  to get,
\[
\frac{ u(E_k\cap I_j^k)}{u(I_j^k)}\le c_1  e^{- c_2 l}.
\]
\end{proof}
Now return to the proof. Fix $l$, form the principal cubes for $\cup_{k\ge N}\Gamma_{l,k}$: let $\mathcal P_0^l$ be the maximal cubes in $\cup_{k\ge N}\Gamma_{l,k}$, then for $m\ge 0$, if $I_s^t\in \mathcal P_m^l$, we say $I_j^k\in \mathcal P_{m+1}^l$ if $I_j^k$ is maximal (in the sense of inclusion) in $ \mathcal D(I_s^t)$ such that
\[
\langle u\rangle_{I_{j}^k}> 2 \langle u\rangle_{I_s^t}
\]
Denote $\mathcal P^l=\cup_{m\ge 0}\mathcal P_m^l$ and $\pi(Q)$ is the minimal principal cube which contains $Q$. We have
\begin{align*}
\sum_k\sum_{l\ge 0}\sum_{I_j^k\in \Gamma_{l,k}}a^{k+1}u(E_k\cap I_j^k)
&\overset{Lemma \ref{lm:decay}} \le \sum_{l\ge 0}c_1  e^{- c_2 l} a^{1-l} \sum_k\sum_{I_j^k\in \Gamma_{l,k}} \langle v\rangle_{I_j^k} u(I_j^k)\\
&= \sum_{l\ge 0}c_1  e^{- c_2 l} a^{1-l} \sum_k\sum_{I_j^k\in \Gamma_{l,k}} \langle u\rangle_{I_j^k} v(I_j^k)\\
&\le \sum_{l\ge 0}2c_1  e^{- c_2 l}a^{1-l} \sum_{I_s^t\in \mathcal P^l} \langle u\rangle_{I_s^t} \sum_{k,j:\pi(I_j^k)=I_s^t} v(I_j^k)\\
&\overset{Lemma \ref{lm:sparse}}{\lesssim_n} \sum_{l\ge 0}c_1  e^{- c_2 l} a^{-l} [v]_{A_\infty}\sum_{I_s^t\in \mathcal P^l} \langle u\rangle_{I_s^t} v(I_s^t)\\
&\le a \sum_{l\ge 0}c_1  e^{- c_2 l}[v]_{A_\infty}\sum_{I_s^t\in \mathcal P^l} a^t u(I_s^t)\\
&\le a \sum_{l\ge 0}c_1  e^{- c_2 l}[v]_{A_\infty}\sum_{I_s^t\in \mathcal P^l}    \langle g \rangle_{I_s^t}  \,  u(I_s^t)\\
&\le a \sum_{l\ge 0}c_1  e^{- c_2 l} [v]_{A_\infty} \int_{\bbR^n} g(x) \left(\sum_{I_s^t\in \mathcal P^l} \langle u\rangle_{I_s^t} \chi^{}_{I_s^t}(x)\right)dx.
\end{align*}
For fixed $x$, there is a chain of principal cubes which contain $x$, say $I_x^m$. By the definition of principal cubes, $\langle u\rangle_{I_x^m}$ forms a geometric sequence (indeed, this sequence is finite), we have
\[
\sum_{I_s^t\in \mathcal P^l} \langle u\rangle_{I_s^t} \chi^{}_{I_s^t}(x)
= \sum_{0\le m\le m_0} \langle u\rangle_{I_x^m} \le \sum_{0\le m\le m_0} 2^{m-m_0}\langle u\rangle_{I_x^{m_0}}\le 2[u]_{A_1}u(x).
\]
Finally, take the sum over $l$ we obtain
\[
\sum_{k\in \bbZ} \sum_{l\ge 0}\sum_{I_j^k\in \Gamma_{l,k}} a^{k+1}u(E_k\cap I_j^k)\le c_{n,u,v} [u]_{A_1}[v]_{A_\infty}\|g\|_{L^1(u)}.
\]

\subsection{The case $l=-1$}
In general, this case follows the strategy in \cite{Sa}. But instead of using $A_1$ property of $v$, the sparsity of $\Gamma$ plays an important role.
First of all, we define the principal cubes with respect to $u$. Set $\mathcal F_0$ be the maximal cubes in $\Gamma_{-1,N}$. Now for $m\ge 0$, assume that $I_{s,\ell}^t \in \mathcal F_m$, then we say $I_{j,i}^k\in \mathcal F_{m+1}$ if $I_{j,i}^k$ is maximal in $\mathcal D(I_{s,\ell}^t)$ such that
\[
\langle u \rangle_{I_{j,i}^k}> a^{(k-t)\delta}\langle u\rangle_{I_{s,\ell}^t}.
\]Finally, we define $\mathcal F=\cup_{m\ge 0}\mathcal F_m$. We still denote by $\pi(Q)$ the minimal principal cube which contains $Q$. If $\pi(I_{j',i'}^{k'})=I_{s,\ell}^t$, then by the proof of Lemma \ref{lm:sparse}, $k'\ge t$. And by definition,
\[
\langle u \rangle_{I_{j',i'}^{k'}}\le  a^{(k'-t)\delta}\langle u\rangle_{I_{s,\ell}^t}.
\]
We have
\begin{align*}
\sum_{k\in \bbZ}\sum_{I_{j,i}^k\in \Gamma_{-1,k}}  a^{k+1}u(I_{j,i}^k)&\le
a\sum_{k\in \bbZ}\sum_{I_{j,i}^k\in \Gamma_{-1,k}} \langle v\rangle_{I_{j,i}^k} u(I_{j,i}^k)\\
&\le a \sum_{I_{s,\ell}^t\in \mathcal F}\langle u\rangle_{I_{s,\ell}^t} \sum_{k,j,i: \pi(I_{j,i}^k)=I_{s,\ell}^t}a^{(k-t)\delta}v(I_{j,i}^k)\\
&=a \sum_{I_{s,\ell}^t\in \mathcal F}\langle u\rangle_{I_{s,\ell}^t} \sum_{k\ge t}a^{(k-t)\delta}\sum_{j,i: \pi(I_{j,i}^k)=I_{s,\ell}^t}v(I_{j,i}^k).
\end{align*}
By the sparsity,
\[
\sum_{j,i: \pi(I_{j,i}^k)=I_{s,\ell}^t}|I_{j,i}^k|\le (\frac {2^n}{a})^{(k-t)}|I_{s,\ell}^t|.
\]It follows from Lemma \ref{RHA1} that
\[
\sum_{j,i: \pi(I_{j,i}^k)=I_{s,\ell}^t}v(I_{j,i}^k)\le  (\frac {2^n}{a})^{\frac{k-t}{2\tau_n[v]_{A_\infty}}} 2v(I_{s,\ell}^t).
\]
Let $\delta=1/{(c_n'[v]_{A_\infty})}$, where $c_n'$ is a sufficient large constant which depends only on dimension. Then we have
\[
\sum_{k\ge t}a^{(k-t)\delta}\sum_{j,i: \pi(I_{j,i}^k)=I_{s,\ell}^t}v(I_{j,i}^k)\le c_n [v]_{A_\infty} v(I_{s,\ell}^t).
\]
It remains to estimate
\begin{align*}
\sum_{I_{s,\ell}^t\in \mathcal F}\langle u\rangle_{I_{s,\ell}^t}v(I_{s,\ell}^t)&\le
 \sum_{I_{s,\ell}^t \in \mathcal F} a^{t+1}u(I_{s,\ell}^t)\\
 &\le a \sum_{I_{s,\ell}^t \in \mathcal F} \langle g\rangle_{I_{s}^t} u(I_{s,\ell}^t)\\
 &= a \int_{\bbR^n} g(x) \left( \sum_{I_{s,\ell}^t \in \mathcal F} |I_s^t|^{-1} u(I_{s,\ell}^t)\chi^{}_{I_s^t}(x)\right)dx.
\end{align*}
We need to prove that
\begin{equation}\label{eq:goal}
\sum_{I_{s,\ell}^t \in \mathcal F} |I_s^t|^{-1} u(I_{s,\ell}^t)\chi^{}_{I_s^t}(x)\le c_n [u]_{A_1}^2[v]_{A_\infty} u(x).
\end{equation}
To this end, let's fix $x$. Since for fixed $t$, there is at most one $I_s^t$ (keep in mind $\langle v \rangle_{I_s^t}\le a^t$) contains $x$, if such $I_s^t$ exists, we denote it by $I^t$. Denote
\[
G= \{I^t: I^t\ni x\}
\]and form the principal cubes for $G$: set $\mathcal G_0=\{I^{k_0}\}$ to be the maximal cube in $G$ ($I^{k_0}$ exists since $k\ge N$). Then if $I^{k_m}\in \mathcal G_m$, we say $I^{k_{m+1}}\in \mathcal G_{m+1}$ if $I^{k_{m+1}}\subset I^{k_m}$ is maximal such that $\langle u\rangle_{I^{k_{m+1}}}>2 \langle u\rangle_{I^{k_m}}$.
Now we have
\begin{align*}
\sum_{I_{s,\ell}^t \in \mathcal F} |I_s^t|^{-1} u(I_{s,\ell}^t)\chi^{}_{I_s^t}(x)
&= \sum_{I_{s,\ell}^t \in \mathcal F} |I^t|^{-1} u(I_{s,\ell}^t)\\
&= \sum_{I_{s,\ell}^t \in \mathcal F} \langle u\rangle_{I^t} \frac{ u(I_{s,\ell}^t)}{u(I^t)}\\
&\le \sum_{m}2\langle u\rangle_{I^{k_m}}\sum_{\substack{I^t\in G\\ k_m\le t<k_{m+1}}}\sum_{s,\ell:I_{s,\ell}^t \in \mathcal F}  \frac{ u(I_{s,\ell}^t)}{u(I^t)}.
\end{align*}
If we have
\begin{equation}\label{eq:goal1}
\sum_{\substack{I^t\in G\\ k_m\le t<k_{m+1}}}\sum_{s,\ell:I_{s,\ell}^t \in \mathcal F}  \frac{ u(I_{s,\ell}^t)}{u(I^t)}\le c_n [u]_{A_1}[v]_{A_\infty},
\end{equation}
then \eqref{eq:goal} follows immediately. So it remains to prove \eqref{eq:goal1}. Denote by $I_{j,i}^{k_m}\subset I^{k_m}$ the cube which contains $I_{s,\ell}^t$ (again by the proof of Lemma \ref{lm:sparse}, $I_{j,i}^{k_m}$ exists). Notice that $I_{j,i}^{k_m}$ not necessary belongs to $\mathcal F$. We assume $I_{j',i'}^{k'}=\pi(I_{j,i}^{k_m})\in \mathcal F$. We have
\begin{align*}
\langle u\rangle_{I_{s,\ell}^t}&> a^{(t-k')\delta}\langle u\rangle_{I_{j',i'}^{k'}},\\
\langle u\rangle_{I_{j,i}^{k_m}}&\le a^{(k_m-k')\delta}\langle u\rangle_{I_{j',i'}^{k'}}.
\end{align*}
It follows that
\begin{align*}
\langle u\rangle_{I_{s,\ell}^t}> a^{(t-k_m)\delta}\langle u\rangle_{I_{j,i}^{k_m}}\ge a^{(t-k_m)\delta}\text{ess}\inf_{y\in I_{j,i}^{k_m}}u(y)&\ge a^{(t-k_m)\delta}\text{ess}\inf_{y\in I^{k_m}}u(y)\\
&\ge \frac{a^{(t-k_m)\delta}}{[u]_{A_1}}\langle u\rangle_{I^{k_m}}\\
&\ge \frac{a^{(t-k_m)\delta}}{2[u]_{A_1}}\langle u\rangle_{I^{t}}.
\end{align*}
Therefore, for a.e. $y\in I_{s,\ell}^t$,
\[
u(y)\ge \frac{a^{(t-k_m)\delta}}{2[u]_{A_1}^2}\langle u\rangle_{I^{t}}:=\lambda.
\]
Thus,
\begin{align*}
\sum_{s,\ell:I_{s,\ell}^t \in \mathcal F}    u(I_{s,\ell}^t)&\le u(\{y\in I^t: u(y)>\lambda\}) \\
&\overset{Lemma \ref{RHA1}}{\le} 2 u(I^t)\left(\frac{|\{y\in I^t: u(y)>\lambda\}|}{|I^t|} \right)^{\frac{1}{2\tau_n [u]_{A_1}}}\\
&\le 2 u(I^t)\left( \lambda^{-1} \langle u\rangle_{I^{t}}\right)^{\frac{1}{c_n [u]_{A_1}}}\\
&\lesssim_n u(I^t) a^{\frac{t-k_m}{c_n [u]_{A_1}[v]_{A_\infty}}}.
\end{align*}
Finally, take the sum over $t$ we conclude the proof of \eqref{eq:goal1}.

\section{Proof of Corollary \ref{Cor2}} \label{proof Cor 2}

The proof of Corollary \ref{Cor2} follows from the following result which can be found in \cite{CMP1}.

\begin{Lemma}\label{theor:extrapol}
Let  $\mathcal F$ be a family of ordered pairs of non-negative, measurable functions $(f,g)$. Let $p_0\in (0,\infty)$ such that for every $w\in A_\infty$,
\begin{equation}\label{p0}
\int_{\mathbb R^n} f(x)^{p_0}\, w(x)\,d x
\le
c\,\int_{\mathbb R^n} g(x)^{p_0}\, w(x)\,d x, \qquad (f,g)\in \mathcal F,
\end{equation}
for all $(f,g)\in \mathcal F$ such that the left-hand side is finite, and where $c$ depends only on the $A_\infty$ constant of $w$.
Then for all $p,q\in (0,\infty)$, and $w\in A_{\infty}$, there is a constant $c$ depending on the $A_{\infty}$ constant of $w$ such that,
\begin{equation}\label{Lp,s:v-v}
\Big\| \Big( \sum_j (f_j)^q \Big)^\frac1q \Big\|_{L^{p}(w)}
\le
c\, \Big\| \Big( \sum_j ( g_j)^q \Big)^\frac1q
\Big\|_{L^{p}(w)},\qquad \{(f_j,g_j)\}_j\subset \mathcal F.
\end{equation}
\end{Lemma}

Now, the hypothesis \eqref{coifman-fefferman} is satisfied for some $p_0$, namely
\begin{equation}
\int_{\mathbb{R}^{n}} |Tf(x)|^{p_0}\, w(x)\,d x
\leq
c\,\int_{\mathbb{R}^{n}} Mf(x)^{p_0}\, w(x)\,d x,
\end{equation}
and hence by Lemma \ref{theor:extrapol}, for all $0<p,q<\infty$, and $w\in A_{\infty}$,
\begin{equation}
\Big\| \Big( \sum_j |T(f_j)|^q \Big)^\frac1q \Big\|_{L^{p}(w)}
\le
C\, \Big\| \Big( \sum_j ( Mf_j)^q \Big)^\frac1q
\Big\|_{L^{p}(w)},\qquad
\end{equation}
for any for any vector function $f=\{f_j\}_j$ such that the left hand side is finite. Now in the case $q>1$ we can be more specific since it was observed in \cite{CGMP} the following pointwise estimate.  Let $1<q<\infty$ and $0<\delta<1$, then there exists a constant
$c>0$ depending on $q,\delta,n$ such that for any vector function $f=\{f_j\}_j$
\begin{equation}\label{puntualinfty}
M^\#_\delta \Big(\overline{M}_{q} f\Big)(x) \leq
c\, M(\|f\|_{\ell^q})(x) \qquad x \in \mathbb R^n,
\end{equation}
using the notation \, $ \overline{M}_qf(x) = \left(\sum_{i=1}^\infty Mf_j(x)^q\right)^{1/q}.$\, Hence, for $q>1$ and $p \in (0,\infty)$,
\begin{equation}
\Big\| \Big( \sum_j |T(f_j)|^q \Big)^\frac1q \Big\|_{L^{p}(w)}
\le
c\, \Big\| M(\|f\|_{\ell^q})
\Big\|_{L^{p}(w)},\qquad
\end{equation}
by using  \eqref{FS-sharpMF}.  We can apply now Lemma \ref{extrapolation}.  Indeed,  hypthesis \eqref{extraphyp} is satisfied choosing $f$ as $\Big( \sum_j |T(f_j)|^q \Big)^\frac1q$ and $g$  as $M(\|f\|_{\ell^q})$. Hence, if  $u\in A_1$ and $v\in A_\infty$,
$$
\Big\|  \frac{  \Big( \sum_j |T(f_j)|^q \Big)^\frac1q  }{v }    \Big\|_{L^{1,\infty}(uv)}\le c\, \|   \frac{   M(\|f\|_{\ell^q})  }{v }    \|_{L^{1,\infty}(uv)}
$$
with constant depending on the $A_1$ constant of $u$ and the $A_{\infty}$ constant of $v$. This finishes the proof of Corollary \ref{Cor2} after applying  Theorem \ref{thm:main}.

\section{Quantitative estimates}

In the main theorem of this paper, Theorem \ref{thm:main}, we show that the operator $f\to \frac{ M(fv)} {v}$ is bounded from
$L^1(uv)$  to $L^{1,\infty}(uv)$ and this bound depends on the $A_1$ constant of $u$ and the $A_{\infty}$ constant of $v$. For many reasons it would be very desirable to find a more precise bound.  This is the purpose of this section, namely to quantify this bound.

\subsection{Dyadic maximal functions}

As in the proof of Theorem \ref{thm:main}, we reduce matters to the dyadic maximal function.
We prove the following result.
\begin{Theorem}\label{thm:quantitativemd}
Suppose that $v\in A_p$ and $u\in A_1$, where $p> 1$.  Then
\begin{align*}
\Big\|&\frac{ M(fv)} {v}\Big\|_{L^{1,\infty}(uv)}\\&\le c_n  [u]_{A_1}[v]_{A_\infty}([u]_{A_1}[v]_{A_\infty}+[u]_{A_\infty} \max\{p, 1+\log ([v]_{A_p}+1)\}  )\|f\|_{L^1(uv)}.
\end{align*}
\end{Theorem}
\begin{proof}
The proof is essentially given in the proof of Theorem \ref{thm:main}, here we only track the dependence on the constant.  Following the same notation, it is easy to check that
\[
\sum_{k\in \bbZ}\sum_{I_{j,i}^k\in \Gamma_{-1,k}}  a^{k+1}u(I_{j,i}^k)\le c_n [u]_{A_1}^2[v]_{A_\infty}^2 \|g\|_{L^1(u)}.
\]
For the remaining term, notice that, for any $v\in A_p$, we have $v\in A_q$ for any $q\ge p$, moreover $[v]_{A_q}\le[v]_{A_p}$. So the same calculations give us
\[
\frac{|E_k\cap I_j^k|}{|I_j^k|}\le a_{}^{\frac{1-l}{q-1}}[v]_{A_q}^{\frac 1{q-1}}.
\]
Let $q= \max\{p, 1+\log ([v]_{A_p}+1)\}$, then
\[
[v]_{A_q}^{\frac 1{q-1}}\le [v]_{A_p}^{\frac{1}{\log ([v]_{A_p}+1)}}\le e.
\]
By Lemma \ref{RHA1}, we obtain
\[
\frac{u(E_k\cap I_j^k)}{u(I_j^k)}\le 2 (ea_{}^{\frac{1-l}{q-1}})^{\frac 1{2\tau_n [u]_{A_\infty}}}.
\]
Then following the same arguments we conclude that
\begin{align*}
\sum_{k\in \bbZ}\sum_{l\ge 0}&\sum_{I_j^k\in \Gamma_{l,k}} a^{k+1}u(E_k\cap I_j^k)\\
&\le c_n  \max\{p, 1+\log ([v]_{A_p}+1)\} [u]_{A_1}[u]_{A_\infty}[v]_{A_\infty}\|g\|_{L^1(u)}.
\end{align*}
\end{proof}

\begin{Remark}
We remark that in the special case of $u=1$, our arguments already give us
\[
\sum_{k\in \bbZ}\sum_{l\ge 0} \sum_{I_j^k\in \Gamma_{l,k}} a^{k+1}|E_k\cap I_j^k|
 \le c_n  \max\{p, 1+\log ([v]_{A_p}+1)\}  [v]_{A_\infty}\|g\|_{L^1(u)}.
\]
For the remaining term, let $\mathcal K$ denote the maximal cubes in $\cup_{k\ge N} \Gamma_{-1,k}$. Then
\begin{align*}
\sum_{k\in \bbZ}\sum_{I_{j,i}^k\in \Gamma_{-1,k}}  a^{k+1}|I_{j,i}^k|&\le a\sum_{k\in \bbZ}\sum_{I_{j,i}^k\in \Gamma_{-1,k}}  v(I_{j,i}^k)\\
&\le c_n [v]_{A_\infty} \sum_{I_{s,\ell}^t\in \mathcal K} v(I_{s,\ell}^t)\\
&\le c_n [v]_{A_\infty} \sum_{I_{s,\ell}^t\in \mathcal K} a \langle g \rangle_{I_s^t} |I_{s,\ell}^t|\\
&\le c_n' [v]_{A_\infty}\|g\|_{L^1(\bbR^n)},
\end{align*}
the last inequality holds since $I_s^t$ are pairwise disjoint (due to the maximality of $I_{s,\ell}^t$ and Lemma \ref{lm:sparse}). So our technique recovers \cite[Theorem 1.13]{OPR}:

Let $v \in A_p$, $p \geq 1$, then there exists a dimensional constant $c$ such that
$$\left\| \frac{M (fv)}{v} \right\|_{L^{1, \infty}(v)}  \; \leq \; c\,[v]_{A_\infty}\max\{p,\,\log(e+[v]_{A_p})\} \|f\|_{L^1(v)}.$$
\end{Remark}

For the case of $v\in A_1$, there is a  conjecture in \cite{OPR}, which states as follows
\begin{Conjecture}
Let $u, v\in A_1$. Then there exists a dimensional constant $c_n$ such that
\[
\Big\|\frac{ M(fv)} {v}\Big\|_{L^{1,\infty}(uv)}\le c_n  [u]_{A_1}[v]_{A_1}\|f\|_{L^1(uv)}.
\]
\end{Conjecture} In the following, we will give a quantitative bound which is far from the conjecture but still improves the bound given in \cite{OPR}. We also give a positive answer to the conjecture when $u=v$.

\begin{Theorem}\label{thm:uva1}
Suppose that $v\in A_1$ and $u\in A_1$.  Then
\[
\Big\|\frac{ M(fv)} {v}\Big\|_{L^{1,\infty}(uv)}\le c_n  [u]_{A_1}[v]_{A_\infty}([u]_{A_1}[v]_{A_\infty}+  \log  [v]_{A_1}    )\|f\|_{L^1(uv)}.
\]
\end{Theorem}
\begin{proof}
In the case of $v\in A_1$,  we have
\[
a^{k+l}<\langle v \rangle_{I_j^k} \le [v]_{A_1} \text{ess\,inf}_{y \in I_j^k} v(y) \le [v]_{A_1} a^{k+1}.
\]
Then $l\le c_n (1+\log [v]_{A_1})$ and the result follows.
\end{proof}

\begin{Theorem}
Suppose that $u\in A_1$, then
\[
\Big\|\frac{M (fu)}{u} \Big\|_{L^{1,\infty}(u^2)}\le c_n [u]_{A_1}^2\|f\|_{L^{1}(u^2)}.
\]
\end{Theorem}
\begin{proof}
The proof is still following the structure and notations of Theorem \ref{thm:main}. First we consider the case $l\ge 0$. Fix $l$, form the principal cubes for $\cup_{k\ge N}\Gamma_{l,k}$: let $\mathcal P_0^l$ be the maximal cubes in $\cup_{k\ge N}\Gamma_{l,k}$, then for $m\ge 0$, if $I_s^t\in \mathcal P_m^l$, we say $I_j^k\in \mathcal P_{m+1}^l$ if $I_j^k$ is maximal in $ \mathcal D(I_s^t)$ such that
\[
\langle u\rangle_{I_{j}^k}> 2 \langle u\rangle_{I_s^t}
\]
Denote $\mathcal P^l=\cup_{m\ge 0}\mathcal P_m^l$ and $\pi(Q)$ is the minimal principal cube which contains $Q$. We have
\begin{align*}
\sum_k\sum_{I_j^k\in \Gamma_{l,k}}a^{k+1}u(E_k\cap I_j^k)
&\le a^{1-l} \sum_k\sum_{I_j^k\in \Gamma_{l,k}} \langle u\rangle_{I_j^k} u(E_k\cap I_j^k)\\
&\le 2a^{1-l} \sum_{I_s^t\in \mathcal P^l} \langle u\rangle_{I_s^t} \sum_{k,j:\pi(I_j^k)=I_s^t} u(E_k\cap I_j^k)\\
&\le 2a^{1-l}  \sum_{I_s^t\in \mathcal P^l} \langle u\rangle_{I_s^t} u(I_s^t)\\
&\le 2a^2  \sum_{I_s^t\in \mathcal P^l} a^t u(I_s^t)\\
&\le 2a^2 \int_{\bbR^n} g(x) \left(\sum_{I_s^t\in \mathcal P^l} \langle u\rangle_{I_s^t} \chi^{}_{I_s^t}(x)\right)dx\\
&\le c_n[u]_{A_1} \|g\|_{L^1(u)}.
\end{align*}
Finally, take the sum over $0\le l \le c_n (1+\log [v]_{A_1})$ we obtain
\[
\sum_{k\in \bbZ} \sum_{l\ge 0}\sum_{I_j^k\in \Gamma_{l,k}} a^{k+1}u(I_j^k)\le c_n [u]_{A_1}(\log [u]_{A_1}+1)\|g\|_{L^1(u)}.
\]

It remains to treat the case $l =-1$. In this case, we need to estimate
\[
\sum_{k\in \bbZ}\sum_{I_{j,i}^k\in \Gamma_{-1,k}}  a^{k+1}u(I_{j,i}^k)=:\sum_{(j,i,k)\in \Lambda } a^{k+1}u(I_{j,i}^k).\]
Keep in mind that in this case $a^k < \langle u \rangle_{I_{j,i}^k} \le a^{k+1}$. Split the collection
\[\{I_{j,i}^k\}_{(j,i,k)\in \Lambda}:= \cup_b \Pi_b:=  \cup_b \Big\{ I_{j,i}^k: 2^b\le  \frac{\langle u \rangle_{I_{j,i}^k} }{\langle u \rangle_{I_{j}^k}}< 2^{b+1}\Big\}.
\]
Since in this case, $\langle u \rangle_{I_j^k} \le a^k < \langle u \rangle_{I_{j,i}^k}$, we know $b \ge 0$. Also notice that
\begin{align*}
\frac{\sum_{i: I_{j,i}^k \in \Pi_b} |I_{j,i}^k|}{|I_j^k|}\le 2^{-b} \frac{u(\cup _{i: I_{j,i}^k \in \Pi_b} I_{j,i}^k)}{u(I_j^k)}\le 2^{-b+1} \Big(\frac{\sum_{i: I_{j,i}^k \in \Pi_b} |I_{j,i}^k|}{|I_j^k|} \Big)^{\frac 1{2\tau_n [u]_{A_\infty }}},
\end{align*}
then
\begin{equation}\label{newkey}
\frac{\sum_{i: I_{j,i}^k \in \Pi_b} |I_{j,i}^k|}{|I_j^k|}\le  2^{2-b(1+\frac 1{2\tau_n [u]_{A_\infty}})}.
\end{equation}
We also need the following observation. Namely, if $I_{s_i, \ell_i}^{t_i} \in \Pi_b$ such that $I_{s_1}^{t_1}\supsetneq I_{s_2}^{t_2}\supsetneq \cdots$, then
\[
\langle u \rangle_{I_{s_1}^{t_1}}\le 2^{-b} \langle u \rangle_{I_{s_1, \ell_1}^{t_1}} \le 2^{-b} a^{t_1+1-t_m}\langle u \rangle_{I_{s_m, \ell_m}^{t_m}}< 2 a^{1+t_1-t_m} \langle u \rangle_{I_{s_m}^{t_m}}.
\]
With all the above observations, we have
\begin{align*}
\sum_{(j,i,k)\in \Lambda } a^{k+1}u(I_{j,i}^k)&\le a \int g(x) \Big(\sum_{(j,i,k)\in \Lambda } \frac{u(I_{j,i}^k)}{|I_j^k|} \chi_{I_j^k}(x)  \Big) dx\\
&\le a \int g(x) \Big(\sum_{b \ge 0} \sum_{I_{j,i}^k\in \Pi_b}\frac{u(I_{j,i}^k)}{|I_j^k|} \chi_{I_j^k}(x)  \Big) dx\\
&\le 2a \int g(x) \Big(\sum_{b \ge 0} 2^b\sum_{I_{j,i}^k\in \Pi_b}\langle u\rangle_{I_j^k}\frac{|I_{j,i}^k|}{|I_j^k|} \chi_{I_j^k}(x)  \Big) dx
\end{align*}
Now fix $x$, suppose $I^{k_m}:=I_{j_m}^{k_m}$ is the chain such that $I^{k_m}\ni x$ and there exists at least one $I_{j_m, i_m}^{k_m}\in \Pi_b$. Then
\begin{align*}
\sum_{I_{j,i}^k\in \Pi_b}\langle u\rangle_{I_j^k}\frac{|I_{j,i}^k|}{|I_j^k|} \chi_{I_j^k}(x)&= \sum_m \langle u \rangle_{I^{k_m}}\sum_{i_m} \frac{|I_{j_m, i_m}^{k_m}|}{|I^{k_m}|}\le 2^{2-b(1+\frac 1{2\tau_n [u]_{A_\infty}})} \sum_m \langle u \rangle_{I^{k_m}}\\
&\le c_n 2^{2-b(1+\frac 1{2\tau_n [u]_{A_\infty}})} Mu(x).\end{align*}
Finally, take the summation over $b$ we conclude that
\[
\sum_{(j,i,k)\in \Lambda } a^{k+1}u(I_{j,i}^k)\le c_n [u]_{A_1}[u]_{A_\infty}\|g\|_{L^1(u)}.
\]
\end{proof}

One might be also interested in the quantitative bound of the case $v\in A_\infty$. To this end, we need the following quantitative embedding result.
\begin{Lemma}\cite{HaPa}\label{lm:embedding}
Let $w\in A_\infty$. Then there exists dimensional constant $c_n$ such that $w\in A_p$ for $p> e^{c_n[w]_{A_\infty}}$ with $[w]_{A_p}\le e^{e^{c_n[w]_{A_\infty}}}$.
\end{Lemma}
Combining Theorem \ref{thm:quantitativemd} and Lemma \ref{lm:embedding} we obtain the following
\begin{Corollary}
Suppose that $v\in A_\infty$ and $u\in A_1$.  Then
\[
\Big\|\frac{ M(fv)} {v}\Big\|_{L^{1,\infty}(uv)}\le c_n  [u]_{A_1}[v]_{A_\infty}([u]_{A_1}[v]_{A_\infty}+ [u]_{A_\infty}e^{c_n [v]_{A_\infty}} )\|f\|_{L^1(uv)}.
\]
\end{Corollary}

\subsection{Calder\'on-Zygmund operators}
In this section, we shall give a quantitative estimate of the following inequality
\begin{equation}\label{eq:CZO}
\left\|\frac{T(fv)}{v} \right\|_{L^{1,\infty}(uv)}\le c_{u,v} \|f\|_{L^1(uv)},
\end{equation}
where $T$ is a Calder\'on-Zygmund operator and $u,v \in A_1$.
 Essentially, the proof will follow the idea in \cite{CMP2}. However, we will make slight changes to give a quantitative relation between the Calder\'on-Zygmund operators and maximal operators. For $u,v\in A_1$, it is easy to check that $vu^{1-p}\in A_p$ with
\[
[vu^{1-p}]_{A_p}\le [v]_{A_1}[u]_{A_1}^{p-1}.
\]
Define $$S(f)(x)= \frac{M(fu)(x)}{u(x)}.$$
Observe that $\|S(f)\|_{L^{\infty}(uv)}\le [u]_{A_1}\|f\|_{L^{\infty}(uv)}$, and that
\[
\| S(f)\|_{L^{p,1}(uv)}\le \| S(f)\|_{L^{p}(uv)}=\| M(fu)\|_{L^p(vu^{1-p})} \le c_n p'[v]_{A_1}^{\frac 1{p-1}}[u]_{A_1},
\]where the last step is due to Buckley \cite{B}.
By interpolation (see e.g. \cite[Proposition A.1]{CMP2}), for $p<q<\infty$, we obtain
\begin{equation}
\|S(f)\|_{L^{q,1}(uv)}\le 2^{\frac 1q}\Big(c_n p'[v]_{A_1}^{\frac{1}{p-1}}[u]_{A_1}(\frac 1p-\frac 1q)^{-1}+ [u]_{A_1}  \Big)\|f\|_{L^{q,1}(uv)}
\end{equation}
Let $p=\log ({[v]_{A_1}}+e)$, then for any
\[
q\ge 2[u]_{A_1}\log ({[v]_{A_1}}+e),
\]
one can check
\[
\|S(f)\|_{L^{q,1}(uv)}\le c_n' [u]_{A_1}\log ({[v]_{A_1}}+e)\|f\|_{L^{q,1}(uv)}.
\]
We denote
\[
K_0:= c_n' [u]_{A_1}\log ({[v]_{A_1}}+e)
\]
and we follow the Rubio de Francia algorithm:
\[
\mathcal Rh(x):= \sum_{k=0}^\infty \frac{S^k(h)(x)}{2^k K_0^k}.
\]
Easily we can check
\begin{enumerate}
\item $h(x) \le \mathcal R h(x)$;\\
\item $\|\mathcal Rh\|_{L^{r',1}(uv)}\le 2\|h\|_{L^{r',1}(uv)}$;\\
\item $[(\mathcal R h) u]_{A_1}\le 2K_0$.
\end{enumerate}
Here $r$ is sharp reverse H\"older constant of $(\mathcal R h) u$, equivalently, $r'\simeq C_n K_0$   by Lemma \ref{RHA1}.
Finally,  following the argument in \cite{CMP2}, we have by duality of the Lorentz spaces and for some parameter $r>1$ to be chosen
\begin{align*}
\Big\|\frac{T(fv)}{v}\Big\|_{L^{1,\infty}(uv)}&=\Big\|\Big(\frac{T(fv)}{v}\Big)^{\frac 1r}\Big\|_{L^{r,\infty}(uv)}^r
\\
&=\sup_{h:\|h\|_{L^{r',1}(uv)}=1}\left(\int (T(fv))^{\frac 1r} u(x)v(x)^{\frac 1{r'}}h(x) dx\right)^r\\
&\le \sup_{h:\|h\|_{L^{r',1}(uv)}=1}\left(\int (T(fv))^{\frac 1r} (\mathcal R h)u(x)v(x)^{\frac 1{r'}} dx\right)^r
\end{align*}
Since $(\mathcal R h)u\in A_1$ and $v\in A_1$, we have
\begin{align*}
\frac 1{|Q|}\int_Q (\mathcal R h)u(x)v(x)^{\frac 1{r'}} &\le \Big(\frac 1{|Q|}\int_Q ((\mathcal R h)u)^r\Big)^{\frac 1r}\Big(\frac 1{|Q|} \int_Q v \Big)^{\frac 1{r'}}\\
&\le 2\Big(\frac 1{|Q|}\int_Q (\mathcal R h)u \Big)\Big(\frac 1{|Q|} \int_Q v \Big)^{\frac 1{r'}}\\
&\le 4K_0 e \inf_{x\in Q} (Rh)(x)u(x) v(x)^{\frac 1{r'}},
\end{align*}
and hence $(Rh)u v^{\frac 1{r'}}\in A_1$ with a constant $[(Rh)u v^{\frac 1{r'}}]_{A_1} \leq 4 K_0$. We use now a more precise Coifman-Fefferman estimate like \eqref{coifman-fefferman} for $T$: let $p\in (0,\infty)$ and let $w\in A_\infty$, then
$$
\|Tf\|_{L^{p}(w)} \leq c_{p,T}   [w]_{A_{\infty}}\, \|M f\|_{L^{p}(w)}
$$
for any smooth function such that the left-hand side is finite (the proof in \cite{LOP3} of Lemma 2.1 can be adapted to this situation). Then, since $r>1$, we continue with

\begin{align*}
\int (T(fv))^{\frac 1r} (\mathcal R h)u(x)v(x)^{\frac 1{r'}} dx&\le c_{n,T} K_0^{1/r}  \int (M(fv))^{\frac 1r} (\mathcal R h)u(x)v(x)^{\frac 1{r'}} dx\\
&\le c_{n,T} K_0 \Big\| \frac{M(fv)}{v}\Big\|_{L^{1,\infty}(uv)}^{\frac 1r} \|\mathcal R h\|_{L^{r',1}(uv)}\\
&\le 2c_{n,T} K_0 \Big\| \frac{M(fv)}{v}\Big\|_{L^{1,\infty}(uv)}^{\frac 1r},
\end{align*}
where we have used H\"older's inequality within the context  of Lorentz spaces. Altogether, we obtain
\begin{Theorem}\label{thm:TtoM}
Let $T$ be a Calder\'on-Zygmund operator and $u,v\in A_1$. Then
\[
\Big\|\frac{T(fv)}{v}\Big\|_{L^{1,\infty}(uv)}\le c_{n,T} [u]_{A_1}\log([v]_{A_1}+e)\Big\|\frac{M(fv)}{v}\Big\|_{L^{1,\infty}(uv)}.
\]
\end{Theorem}

Combining Theorems \ref{thm:TtoM} and \ref{thm:uva1}, we obtain the following results, the first one recovers \cite[Theorem 1.4]{LOP2}.
\begin{Corollary}
Let $T$ be a Calder\'on-Zygmund operator and $v\in A_1$. Then
\[
\Big\|\frac{T(fv)}{v}\Big\|_{L^{1,\infty}(v)}\le c_{n,T} [v]_{A_1}\log([v]_{A_1}+e)\|f\|_{L^{1}(v)}.
\]
\end{Corollary}

\begin{Corollary}\label{cor:CZO}
Let $T$ be a Calder\'on-Zygmund operator and $u,v\in A_1$. Then
\[
\Big\|\frac{T(fv)}{v}\Big\|_{L^{1,\infty}(uv)}\le c_{n,T}c [u]_{A_1}^3 [v]_{A_\infty}([v]_{A_\infty}+\log [v]_{A_1})(\log[v]_{A_1}+1)\|f\|_{L^{1}(v)}.
\]
\end{Corollary}
We believe that Corollary \ref{cor:CZO} is not sharp since we get the estimate through Theorem \ref{thm:TtoM}, which is not a sharp way (one can check this fact in \cite{LOP1, LOP3, HyPe1} for the case of $v=1$). So there is still an open question, namely, how to obtain such estimate directly, without using extrapolation.

\end{document}